\documentclass{article}

\usepackage{anysize}
\usepackage{amsthm}
\usepackage{amsmath}
\usepackage{amssymb}
\usepackage{amsfonts}
\usepackage{amsxtra} 
\usepackage{mathrsfs}
\usepackage{color}
\usepackage{graphicx}
\usepackage[all]{xy}

\marginsize{3.5cm}{3.5cm}{3.5cm}{3.5cm}

\title{Filter-Laver Measurability}
\author{Yurii Khomskii\footnote{This research was partially done whilst the author was a visiting fellow 
at the Isaac Newton Institute for Mathematical Sciences in the programme 
\emph{Mathematical, Foundational and Computational Aspects of the Higher 
Infinite} (HIF) funded by EPSRC grant EP/K032208/1.}} 

\newcommand{\p}{\medskip \noindent}

\newcommand{\cont}{{2^{\aleph_0}}}
\newcommand{\cc}{{^\frown}}

\newcommand{\till}{{\upharpoonright}}
\newcommand{\thru}{{\uparrow}}

\newcommand{\embeds}{\lhook\joinrel\relbar\joinrel\rightarrow}

\newcommand{\G}{\boldsymbol{\Gamma}}
\newcommand{\SIGMA}{\boldsymbol{\Sigma}}
\newcommand{\DELTA}{\boldsymbol{\Delta}}
\newcommand{\PI}{\boldsymbol{\Pi}}

\newcommand{\ww}{\omega^\omega}
\newcommand{\wlw}{\omega^{<\omega}}
\newcommand{\dw}{2^\omega}
\newcommand{\dlw}{2^{<\omega}}
\newcommand{\wupw}{\omega^{\uparrow \omega}}
\newcommand{\wuw}{[\omega]^\omega}

\newcommand{\stem}{{\rm stem}}
\newcommand{\Succ}{{\rm Succ}}

\newcommand{\IP}{\mathbb{P}}

\newcommand{\B}{\mathcal{B}}
\newcommand{\F}{\mathcal{F}}
\newcommand{\ran}{{\rm ran}}
\newcommand{\dom}{{\rm dom}}

\newcommand{\s}[1]{\medskip \noindent \textbf{#1}}

\newcommand{\N}{\mathcal{N}}

\newcommand{\IL}{\mathbb{L}}

\newcommand{\ID}{\mathbb{D}}

\newcommand{\IC}{\mathbb{C}}

\newcommand{\add}{{\rm add}}

\newcommand{\cof}{{\rm cof}}

\newcommand{\bb}{\mathfrak{b}}

\newtheorem{Thm}{Theorem}[section]

\newtheorem{Lem}[Thm]{Lemma}
\newtheorem{Cor}[Thm]{Corollary}

\newtheorem{Question}[Thm]{Question}
\newtheorem{Fact}[Thm]{Fact}

\theoremstyle{definition}
\newtheorem{Def}[Thm]{Definition}
\newtheorem{Notation}[Thm]{Notation}
\newtheorem{Remark}[Thm]{Remark}

\newcommand{\I}{\mathcal{I}}

\renewcommand{\F}{{{F}}}

\newcommand{\Cof}{{\sf Cof}}

\newcommand{\Fin}{{\sf Fin}}
\newcommand{\ILF}{{\mathbb{L}_{\F}}}
\newcommand{\ILFF}{{\mathbb{L}_{{\F}^+}}}
\newcommand{\ILG}{{\mathbb{L}_{G}}}
\newcommand{\ILGG}{{\mathbb{L}_{{G}^+}}}

\begin{document}

\maketitle

\begin{abstract} We study $\sigma$-ideals and regularity properties related to the ``filter-Laver'' and ``dual-filter-Laver'' forcing partial orders. An important innovation which enables this study is a dichotomy theorem proved recently by Miller \cite{MillerMesses}. 
\end{abstract}

\section{Introduction}

In this paper, $\F$ will always be a filter on $\omega$ (or a suitable countable set). We will use $\F^-$ to refer to the ideal of all $a \subseteq \omega$ such that $\omega \setminus a \in F$, and $\F^+$ to the collection of $a \subseteq \omega$ such that $a \notin F^-$. 
$\Cof$ and $\Fin$ denote  the filter of {cofinite subsets} of $\omega$ and the ideal of finite subsets of $\omega$, respectively.

\begin{Def} An \emph{$\F$-Laver tree} is a tree $T \subseteq \wlw$ such that for all $\sigma \in T$ extending $\stem(T)$, $\Succ_T(\sigma) \in F$. An \emph{$\F^+$-Laver-tree} is a tree $T \subseteq \wlw$ such that for all $\sigma \in T$ extending $\stem(T)$, $\Succ_T(\sigma) \in \F^+$. We use $\ILF$ and $\ILFF$ to denote the partial orders of $F$-Laver and $\F^+$-Laver trees, respectively, ordered by inclusion. \end{Def}

If $\F = \Cof$ then $\ILFF$ is the standard \emph{Laver forcing} $\IL$, and $\ILF$ is (a version of) the standard \emph{Hechler forcing} $\ID$.
Both $\ILF$ and $\ILFF$ have been used as forcing notions in the literature, see, e.g., \cite{Groszek}. As usual, the generic real added by these forcings can be defined as the limit of the stems of conditions in the generic filter. It is easy to see that in both cases, this generic real is dominating. It is also known that if $F$ is not an ultrafilter, then $\IL_F$ adds a Cohen real, and if $F$ is an ultrafilter, then $\IL_F$ adds a Cohen real if and only if $F$ is not a \emph{nowhere dense ultrafilter} (see Definition \ref{nwdultrafilter}). Moreover, $\ILF$ is $\sigma$-centered and hence satisfies the ccc, and it is known that $\ILFF$ satisfies Axiom A (see \cite[Theorem]{Groszek} and Lemma \ref{b} (3)).

\bigskip In this paper, we consider  $\sigma$-ideals and regularity properties naturally related to $\IL_F$ and $\ILFF$, and study the regularity properties for sets in the low projective hierarchy, following ideas from \cite{BrLo99, Ik10, KhomskiiThesis}. An important technical innovation is a dichotomy theorem proved recently by Miller in \cite{MillerMesses} (see Theorem \ref{dichotomy}), which allows us to simplify the  $\sigma$-ideal for $\ILFF$ when restricted to Borel sets, while having a $\SIGMA^1_2$ definition regarding the membership of Borel sets in it.

\bigskip One question may occur to the reader of this paper: why are we not considering the \emph{filter-Mathias} forcing alongside the filter-Laver forcing, when clearly the two forcing notions (and their derived $\sigma$-ideals and regularity properties) are closely related? The answer is that, although the basic results from Section 2 do indeed hold for filter-Mathias, there is no corresponding dichotomy theorem like Theorem \ref{dichotomy}. In fact, by a result of Sabok \cite{MarcinSabok}, even the $\sigma$-ideal corresponding to the \emph{standard} Mathias forcing is not a $\SIGMA^1_2$-ideal on Borel sets, implying that even in this simple case, there is no hope of a similar dichotomy theorem. It seems that in the Mathias case, a more subtle analysis is required.

\bigskip

In Section \ref{Sec2} we give the basic definitions and prove some easy properties. In Section \ref{Sec3} we present Miller's dichotomy and the corresponding $\sigma$-ideal. In Section \ref{Sec4} we study direct relationships that hold between the regularity properties regardless of the complexity of $F$, whereas in Section \ref{Sec5} we prove  stronger results under the assumption that $F$ is an analytic filter.

\section{$(\ILF)$- and $(\ILFF)$-measurable sets.} \label{Sec2}

In \cite{Ik10}, Ikegami provided a natural framework for studying $\sigma$-ideals and regularity properties related to tree-like forcing notions, generalising the concepts of \emph{meager} and \emph{Baire property}. This concept proved to be very useful in a number of circumstances, see, e.g., \cite{KhomskiiThesis, LaguzziPaper, KhomskiiLaguzzi1}.

\begin{Def} \label{a} Let $\IP$ be $\ILF$ or $\ILFF$ and let $A \subseteq \ww$.

\begin{enumerate} 
\item $A \in \N_\IP$ iff $\forall T \in \IP \: \exists S \leq T \:([S] \cap A = \varnothing)\}$. 
\item $A \in \I_\IP$ iff $A$ is contained in a countable union of sets in $\N_\IP$.
\item $A$ is \emph{$\IP$-measurable} iff $\forall T \in \IP \: \exists S \leq T \: ([S] \subseteq^* A$ or $[S] \cap A =^* \varnothing)$, where $\subseteq^*$ and $=^*$ stands for ``modulo a set in $\I_\IP$''.
\end{enumerate}
\end{Def}

\begin{Lem}  The collection $\{[T] \mid T \in \ILF\}$ forms a topology base. The resulting topology refines the standard topology and the space satisfies the Baire category theorem $($i.e., $[T] \notin \I_\ILF$ for all $T \in \IL_F)$.\end{Lem}

\begin{proof} Clearly, for all $S,T \in \ILF$ the intersection $S \cap T$ is either empty or an $\ILF$-condition. A basic open set in the standard topology trivially corresponds to a tree in $\ILF$. For the Baire category theorem, let $A_n$ be nowhere dense and, given an arbitrary $T \in \ILF$, build a sequence $T = T_0 \geq T_1  \geq T_2 \geq \dots$  with strictly increasing stems such that $[T_n] \cap A_n = \varnothing$ for all n. Then the limit of the stems is an element in $[T] \setminus \bigcup_n A_n$. \end{proof}

We use $\tau_\ILF$ to denote the topology on $\ww$ generated by $\{[T] \mid T \in \ILF\}$. Clearly $\N_\ILF$ is the collection of $\tau_\ILF$-nowhere dense sets and $\I_\ILF$ the collection of $\tau_\ILF$-meager sets. Moreover, we recall the following fact, which is true in arbitrary topologal spaces (the proof is similar to \cite[Theorem 8.29]{Kechris}): 

\begin{Fact} Let $\mathcal{X}$ be any topological space, and $A \subseteq \mathcal{X}$. Then the following are equivalent: \begin{enumerate}
\item $A$ satisfies the Baire property.
\item For every basic  open $O$ there is a basic open $U \subseteq O$ such that $U \subseteq^* A$ or $U \cap A =^* \varnothing$, where $\subseteq^*$ and $=^*$ refer to ``modulo meager''.
\end{enumerate} In particular, $A \subseteq \ww$ is $\ILF$-measurable iff $A$ satisfies the $\tau_\ILF$-Baire property.
\end{Fact}

What about the dual forcing $\ILFF$? Notice that a topological approach cannot work in general: 

\begin{Lem} The collection $\{[T] \mid T \in \ILFF\}$ generates a topology base iff $F$ is an ultrafilter. \end{Lem}

\begin{proof}   If  $F$ is not an ultrafilter, fix $Z$ such that $Z \in \F^+$ and $(\omega \setminus Z) \in \F^+$ and consider trees $S, T \in \ILFF$ defined so that $\forall \sigma \in S \: (\Succ_S(\sigma) = Z \cup \{0\})$ and $\forall \tau \in T \: (\Succ_T(\tau) = (\omega \setminus Z) \cup \{0\})$. 
\end{proof} 

Instead, to study $\ILFF$, we rely on combinatorial methods familiar from Laver forcing. 
For every $n$, define $\leq_n$ by: $$S \leq_n T \; : \Leftrightarrow \; S \leq T \text{ and } S \cap \omega^{\leq k+n} = T \cap \omega^{\leq k+n},$$ where $k = |\stem(T)|$. 
If $T_0 \geq_0 T_1 \geq_1 \dots$ is a decreasing sequence then $T := \bigcap_n{T_n} \in \ILFF$ and $T \leq T_n$ for every $n$. 

\begin{Lem} \label{b} Let $F$ be a filter on $\omega$. Then: \begin{enumerate}

\item $\ILFF$ has \emph{pure decision}, i.e., for every $\phi$ and every $T \in \ILFF$, there is $S \leq_0 T$ such that $S \Vdash \phi$ or $S \Vdash \lnot \phi$.

\item For all $A \subseteq \ww$, the following are equivalent:
\begin{enumerate}
\item $A \in \N_\ILFF$,
\item $\forall T \in \ILFF \: \exists S \leq_0 T \:( [S] \cap A = \varnothing)$.
\end{enumerate}
\item $\N_\ILFF = \: \I_\ILFF$.
\item For all $A \subseteq \ww$, the following are equivalent:
\begin{enumerate}
\item $A$ is $(\ILFF)$-measurable,
\item $\forall T \in \ILFF \: \exists S \leq T \: ([S] \subseteq A$ or $[S] \cap A = \varnothing)$,
\item $\forall T \in \ILFF \: \exists S \leq_0 T \: ([S] \subseteq A$ or $[S] \cap A = \varnothing)$. \end{enumerate}
\item The collection of $(\ILFF)$-measurable sets forms a $\sigma$-algebra. \end{enumerate} \end{Lem}

\begin{proof} Since many of the arguments here are similar, we prove the first assertion and only sketch the others. 
 $\;$ \

 \begin{enumerate}
 \item Fix $\phi$ and $T$ and let $u := \stem(T)$. For $\sigma \in T$ extending $u$, say: \begin{itemize}
\item $\sigma$ is \emph{positive-good} if $\exists S \leq_0 T \thru \sigma$ such that $S \Vdash \phi$,
\item $\sigma$ is \emph{negative-good} if $\exists S \leq_0 T \thru \sigma$ such that $S \Vdash \lnot \phi$,
\item $\sigma$ is \emph{bad} if neither of the above holds.\end{itemize}
We claim that $u$ is good, completing the proof. Assume that $u$ is bad. Partition  $\Succ_T(u)$ into $Z_0, Z_1$ and $Z_2$ by setting $n \in Z_0$ iff $u \cc \left<n\right>$ is positive-good, $n \in Z_1$ iff $u \cc \left<n\right>$ is negative-good, and $n \in Z_2$ iff  $u \cc \left<n\right>$ is bad. One of the three components must be in $\F^+$. But if it is $Z_0$ then $S:= \bigcup_{n \in Z_0} T \thru (u \cc \left<n\right>) \leq_0 T$ and $S \Vdash \phi$, thus $u$ is positive-good contrary to assumption; likewise, if $Z_1$ is in $\F^+$ then $u$ is negative-good contrary to assumption. Hence, $Z_2$ must be in $\F^+$. Now, for each $n \in Z_2$, use the same argument to obtain an $\F^+$-positive set $Z_{2,2}$ of successors of $u \cc\left<n\right>$ such that for all $m \in Z_{2,2}$, $u \cc \left<n,m\right>$ is bad, and so on. 

\p This way we construct a tree $T^* \leq T$ such that all $\sigma \in T^*$ are bad. But there is a $T^{**} \leq T^*$ deciding $\phi$, which means that $\stem(T^{**})$ is either positive-good or negative-good, leading to a contradiction. 

 \item Let $A \in \N_\ILFF$, fix $T$, and let $u = \stem(T)$. For $\sigma \in T$ extending $u$, say that $\sigma$ is \emph{good} if $\exists S \leq_0 T \thru \sigma$ such that $[S] \cap A = \varnothing$, and $\sigma$ is \emph{bad} otherwise. By the same argument as above we prove that $u$ is good.

 \item Suppose  $A_n \in \N_\ILFF$ for all $n$. 
 Fix $T \in \ILFF$. Clearly it is enough to produce a fusion sequence $T = T_0 \geq_0 T_1 \geq_1 \dots$ such that for all $n$, $[T_n] \cap A_n = \varnothing$. So suppose we have constructed $T_n$. Let $\{u_i \mid i<\omega\}$ enumerate all the nodes in $T_n$ of length $|\stem(T_n) + n|$. For each $u_i$, use $(2)$ to find $S_i \leq_0 T_n \thru u_i$ with $[S_i] \cap A_{n+1} = \varnothing$. Let $T_{n+1} := \bigcup_i S_i$. Then clearly $T_{n+1} \leq_n T_n$ and $[T_{n+1}] \cap A_{n+1} = \varnothing$ as required.


 \item For $(a) \Rightarrow (b)$, use the fact that $\I_\ILFF = \; \N_\ILFF$. For $(b) \Rightarrow (c)$, use the same argument as in $(1)$. 

 \item It suffices to show closure under countable unions. Suppose $A_n$ is $\ILFF$-measurable and fix $T \in \ILFF$. If for one $n$, there is $S \leq T$ with $[S] \subseteq A_n$ then we are done. Otherwise (using the equivalence from (4)) for every $n$, there is $S \leq T$ such that $[S] \cap A_n = \varnothing$. Then an argument like in (3) shows that there is $S \leq T$ such that $[S] \cap \bigcup_n A_n = \varnothing$. \qedhere
\end{enumerate}
\end{proof}


\begin{Remark} \label{strongproper} Note that an argument like in (4) above in fact shows that $\ILFF$ satisfies a stronger form of properness, namely,  for all countable elementary models $M \prec \mathcal{H}_\theta$ and all $T \in \ILFF$, there exists $S \leq T$ such that every $x \in [S]$ is $\ILFF$-generic over $M$. \end{Remark}


Again it is interesting to ask whether any of the ``simplifications'' (1)--(4) from the above Lemma might go through for $\ILF$, too.

\begin{Lem} If we replace $\ILFF$ with $\ILF$ in Lemma \ref{b}, then  the statements (1)--(4) are all equivalent to each other, and equivalent to the statement ``$\F$ is an ultrafilter''. \end{Lem}

\begin{proof} If $\F$ is not an ultrafilter, let $Z$ be such that $Z \in \F^+$ and $(\omega \setminus Z) \in \F^+$, let $A_n := \{x \in \ww \mid \forall m \geq n \: (x(m) \in Z)\}$ and $A = \bigcup_n A_n = \{x \in \ww \mid \forall^\infty m \: (x(m) \in Z)\}$. Also, $x_G$ denotes the $\ILF$-generic real. We leave it to the reader to verify that \begin{itemize}

\item the statement ``$x_{G}(0) \in Z$'' cannot be decided by any $\ILF$-condition with empty stem (falsifying (1)),
\item $Z^\omega \in \N_\ILF$ but for every $T \in \ILF$ with empty stem we have $[T] \cap Z^\omega \neq \varnothing$ (falsifying (2)),
\item $A_n \in \N_\ILF$ for all $n$, but $A \notin \N_\ILF$ (falsifying (3)), and
\item $A$ is $\ILF$-measurable (see Theorem \ref{c}), but for every $T \in \ILF$ we have $[T] \not\subseteq A$ and $[T] \cap A \neq \varnothing$ (falsifying (4)). \qedhere
\end{itemize}

\end{proof}

Thus, the situation can be neatly summarized as follows: when $F$ is \emph{not} an ultrafilter, $\ILF$ generates a topology but does not satisfy properties 1--4  from Lemma \ref{b}, while $\ILFF$ satisfies those properties but does not generate a topology. $\ILF$-measurability is the Baire propery in the $\tau_\ILF$-topology, whereas $\ILFF$-measurability is the ``Marczewski''-property corresponding to the partial order $\ILFF$, and $\I_\ILFF$ is the ``Marczewski''-ideal corresponing to $\ILFF$.

In the interesting scenario when $F$ is an ultrafilter everything coincides, and the ideal $\I_\ILF$ of $\tau_\ILF$-meager sets is the same as the ideal of $\tau_\ILF$-nowhere dense sets. In this context, the ideal has been studied by Louevau in \cite{LouveauIdeal} and is sometimes called the \emph{Louveau ideal}.

\begin{Thm} \label{c} Let $F$ be a filter on $\omega$. Every analytic and co-analytic set $A \subseteq \ww$ is both $\ILF$-measurable and $\ILFF$-measurable. \end{Thm}

\begin{proof} Since $\tau_\ILF$ refines the standard topology on $\ww$, analytic (co-analytic) sets are also analytic (co-analytic) in  $\tau_\ILF$. By classical results, such sets have the $\tau_\ILF$-Baire property. 

\p For $\ILFF$, suppose $A$ is analytic, defined by a $\Sigma^1_1(r)$ formula $\phi$. Let $T \in \ILFF$. Let $S \leq T$ be a stronger condition forcing $\phi(\dot{x}_G)$ or $\lnot \phi(\dot{x}_G)$, without loss of generality the former. Let $M$ be a countable elementary submodel of a sufficiency large $\mathcal{H}_\theta$ with $S, r, F \in M$. 
 By Remark \ref{strongproper}, we can find an $S' \leq S$ such that all $x \in [S']$ are $\ILFF \cap M$-generic over $M$. Then for all such $x$ we have $M[x] \models \phi(x)$. By $\SIGMA^1_1$-absoluteness, $\phi(x)$ is really true. Thus we have $[S'] \subseteq A$. The co-analytic case is analogous. \end{proof}

A different (forcing-free) proof of the second assertion will follow from Theorem \ref{dichotomy}.

\bigskip
\renewcommand{\B}{{\sf Borel}}
From the above it follows that there we have dense embeddings $\ILF \embeds_d \B(\ww)/\I_\ILF$ and $\ILFF \embeds_d \B(\ww)/\I_\ILFF$.

\begin{Def} Let $\G$ be a projective pointclass. The notation $\G(\ILF)$ and $\G(\ILFF)$ abbreviates the propositions ``all sets of complexity $\G$ are $\ILF$-measurable'' and ``all sets of complexity $\G$ are $\ILF$-measurable'', respectively. \end{Def}

The statements $\SIGMA^1_2(\ILF)$ and $\SIGMA^1_2(\ILFF)$ are independent of ZFC, and we will study the exact strength of these statements in Section \ref{Sec4} (for arbitrary $F$) and Section \ref{Sec5} (for definable $F$).

\section{A dichotomy theorem for $\ILFF$} \label{Sec3}

While $\I_\ILF$ is a ccc Borel-generated ideal exhibiting many familiar properties, $\I_\ILFF$ is a ``Marczewski-style'' ideal, which is not Borel-generated and rather difficult to study. The rest of the paper depends crucially on the dichotomy result presented in this section, which simplifies the ideal $\I_\ILFF$ when it is restricted to Borel sets. 
The proof, as well as several key insights, are due to Arnold Miller \cite{MillerMesses}. For motivation, recall the \emph{Laver dichotomy}, originally due to Goldstern et al \cite{GoldsternGame}.

\newcommand{\D}{{\mathcal{D}}}
\newcommand{\FF}{{{\F}^+}}

\begin{Def} \label{aaa} If $f: \wlw \to \omega$ and $x \in \ww$, we say that \emph{$x$ strongly dominates $f$} if $\forall^\infty n \: (x(n) \geq f(x \till n))$. A family $A \subseteq \ww$ is called \emph{strongly dominating} if for every $f: \wlw \to \omega$ there exists $x  \in A$ which strongly dominates $f$. $\D$ denotes the ideal of sets $A$ which are \emph{not} strongly dominating. \end{Def}

It is easy to see that if $T \in \IL$ then $[T] \notin \D$, and the classical result \cite[Lemma 2.3]{GoldsternGame} shows that if $A$ is analytic, then either $A \in \D$ or there is a Laver tree $T$ such that $[T] \subseteq A$. The ideal $\D$ was discovered independently by Zapletal (cf. \cite[Lemma 3.3.]{IsolatingCardinals}) and was studied, among others, in \cite{RepickyDeco, Deco}. Generalising this, we obtain the following definitions:

\begin{Def} Let $F$ be a filter on $\omega$. If $\varphi: \wlw \to F$ and $x \in \ww$, we say that \emph{$x$ $F$-dominates $\varphi$} iff $\forall^\infty n \:(x(n) \in \varphi(x \till n))$. A family $A \subseteq \ww$ is \emph{$F$-dominating} if for every $\varphi: \wlw \to F$ there exists $x  \in A$ which  dominates $\varphi$. $\D_{\FF}$ denotes the ideal of sets $A$ which are not $F$-dominating. In other words:

$$A \in \D_\FF \; :\Longleftrightarrow \; \exists \varphi: \wlw \to F \;\; \forall x \in A \;\; \exists^\infty n \:(x(n) \notin \varphi(x \till n)). $$
\end{Def}

In the above context, the terminology ``$F$-dominates'' might seem inappropriate, but we choose it in order to retain the analogy with Definition \ref{aaa}. Note that $\D = \D_{\Cof^+}$.

\begin{Lem} $\D_\FF$ is a $\sigma$-ideal. \end{Lem}

\begin{proof} Suppose $A_i \in \D_\FF$ for $i < \omega$. Let $\varphi_i$ witness this for each $i$, and define $\varphi$ by setting $\varphi(\sigma) := \bigcap_{i<|\sigma|} \varphi_i(\sigma)$. We claim that $\varphi$ witnesses that $A = \bigcup_{i<\omega} A_i \in \D_\FF$. Pick $x \in A$. There is $i$ such that $x \in A_i$, hence for infinitely many $n$ we have $x(n) \notin \varphi_i(x \till n)$. But if $n > i$ then $\varphi(x \till n) \subseteq \varphi_i(x \till n)$. Therefore, for infinitely many $n$ we also have $x(n) \notin \varphi(x \till n)$. \end{proof}

\begin{Lem} \label{d} Let $A \subseteq \ww$. The following are equivalent:  \begin{enumerate}
\item $A \in \D_\FF$.
\item $\forall \sigma \in \wlw \: \exists T \in \IL_F$ with $\stem(T) = \sigma$, such that $[T] \cap A = \varnothing$. 
\item $\forall S \in \ILF \: \exists T \leq_0 S \: ([S] \cap A = \varnothing)$ \end{enumerate} \end{Lem}

\begin{proof} The equivalence between 2 and 3 is clear  so we prove the equivalence between 1 and 2.

\p First, note that if $\varphi: \wlw \to F$ and $\sigma \in \wlw$, then there is a unique  $T_{\sigma, \varphi} \in \IL_F$ such that $\stem(T_{\sigma,\varphi}) = \sigma$ and $\forall \tau \supseteq \sigma$, $\Succ_{T_{\sigma,\varphi}}(\tau) = \varphi(\tau)$. Conversely, for every $T \in \IL_F$ with $\stem(T) = \sigma$,  there exists a (not unique) $\varphi$ such that $T = T_{\sigma,\varphi}$.

\p Now suppose $A \in \D_\FF$, as witnessed by  $\varphi$, and let $\sigma \in \wlw$.  Then $A \cap [T_{\sigma,\varphi}] = \varnothing$, since if $x \in A \cap [T_{\sigma,\varphi}]$ then $\forall n > |\sigma| \:( x(n) \in \varphi(x\till n))$, contrary to the assumption.

\p Conversely, suppose for every $\sigma$ there is $T_\sigma \in \IL_F$ such that $\stem(T_\sigma) = \sigma$ and $A \cap [T_\sigma] = \varnothing$. For each $\sigma$, let $\varphi_\sigma: \wlw \to F$ be such that $T_\sigma = T_{\sigma, \varphi_\sigma}$. Then define $\varphi: \wlw \to F$ by $$\varphi(\sigma) = \bigcap_{\tau \subseteq \sigma} \varphi_\tau(\sigma).$$
We claim that $\varphi$ witnesses that $A \in \D_\FF$. Let $x \in A$ be arbitrary. Let $\sigma \subseteq x$. Then $x \notin [T_\sigma] = [T_{\sigma ,\varphi_\sigma}]$, hence, there is $n > |\sigma|$ such that $x(n) \notin \varphi_\sigma(x \till n)$. But by definition, since $\sigma \subseteq x \till n$, we have $\varphi(x \till n) \subseteq \varphi_\sigma(x \till n)$. Therefore also $x(n) \notin \varphi(x \till n)$. 
\end{proof}

The following are easy consequences of the above; the proofs are left to the reader.
\begin{Lem} \label{35}

$\;$ \begin{enumerate}

\item $\D_\FF \subseteq \N_\ILF$.
\item $\D_\FF \subseteq \I_\ILFF$ $($in particular, if $T \in \ILFF$ then $[T] \notin \D_\FF)$.
\item If $F$ is an ultrafilter then $\D_\FF = \N_\ILF  = \I_\ILF = \I_\ILFF$.
\item If $F$ is not an ultrafilter then there is a closed witness to $\D_\FF \neq \N_\ILF$.  \end{enumerate}
\end{Lem}




\begin{Thm}[Miller] \label{dichotomy} For every analytic $A$, either $A \in \D_\FF$ or there is $T \in  \ILFF$ such that $[T] \subseteq A$. \end{Thm}

\begin{proof} See the proof of Theorem 3 and the comment after Theorem 8 in \cite{MillerMesses}.\footnote{Here we should also note that Miller's Theorem 3 is, in fact, a direct consequence of Goldstern et al's dichotomy  \cite[Lemma 2.3]{GoldsternGame}. However, the point is that its generalisation to filters does not follow from the proof in \cite{GoldsternGame}, which uses infinite games and determinacy. Miller's proof, on the other hand, uses only classical methods and generalises directly to filters.} We need a slight modification of this proof: rather than talking about trees \emph{with empty stem}, we consider trees with a fixed stem $\sigma$. If $A \notin \D_\FF$, then by Lemma \ref{d} (2), there exists $\sigma \in \wlw$ such that for all $S \in \ILF$ with $\stem(S) = \sigma$, $[S] \cap A \neq \varnothing$. By applying the same argument as in  \cite[Theorem 3]{MillerMesses}, we obtain a $T \in \ILFF$ (with $\stem(T) = \sigma$) such that $[T] \subseteq A$.\end{proof}
 
\begin{Remark} \label{forcingfree} As a direct consequence of this theorem, we obtain an alternative (forcing-free) proof of the second part of Theorem \ref{c}. Namely: let $A$ be analytic and let $T \in \ILFF$ be arbitrary, so $A \cap [T]$ is analytic. If there exists $S \in \ILFF$ with $[S] \subseteq A \cap [T]$   we are done, and if $A \cap [T] \in \D_\FF$,   use Lemma \ref{d} to find a tree $U \in \ILF$ with $\stem(U) = \stem(T)$ and $[U] \cap A \cap [T] = \varnothing$. Notice that $T \cap U \in \ILFF$, so we are done.\end{Remark}

Also, we now have a dense embedding $\ILFF \embeds_d \B(\ww) / \D_\FF$, with $\D_\FF$ being a Borel-generated $\sigma$-ideal which is far easier to study than $\I_\ILFF$. This will be of particular importance in Section \ref{Sec5} where we look at analytic filters.

\section{Direct implications}\label{Sec4}
 
We first look at some straightforward implications between various statements of the form $\G(\ILF)$, $\G(\ILFF)$ and $\G(\IP)$ for other well-known forcings $\IP$. Here $\G$ denotes an arbitrary \emph{boldface pointclass}, i.e., a collection of subsets of $\ww$ closed under continuous pre-images and intersections with closed sets. No further assumptions on the complexity of $\F$ are required.

\medskip Recall the following reducibility relations for filters on a countable set:

\begin{Def} Let $F,G$ be filters on $\dom(F)$ and $\dom(G)$, respectively. We say that: \begin{enumerate}
\item $G$ is \emph{Katetov-reducible} to $F$, notation $G \leq_K F$, if there is a map $\pi: \dom(F) \to \dom(G)$ such that  $a \in G \Rightarrow \pi^{-1}[a] \in F$.
\item $G$ is \emph{Rudin-Keisler-reducible} to $F$, notation $G \leq_{RK} F$, if there is a map $\pi: \dom(F) \to \dom(G)$ such that  $a \in G \Leftrightarrow \pi^{-1}[a] \in F$. \end{enumerate} \end{Def}
 
\begin{Remark} Note that $G \leq_K F$ and $G \leq_{RK} F$ are equivalent to the reducibility relation between ideals (i.e., between $G^-$ and $F^-$). Also, it is clear that if $\pi$ witnesses $G \leq_K F$, then $a \in F^+ \Rightarrow \pi[a] \in G^+$, and if $\pi$ witnesses $G \leq_{RK} F$ then, in addition, $a \in F \Rightarrow \pi[a] \in G$. \end{Remark}

\begin{Notation} We use the following slight abuse of notation: if $F$ is a filter and $a \in F^+$, then $F \till a$ denotes the set $\{b \subseteq a \mid (a \setminus b) \in F^-\}$. In other words, $F \till a$ is the filter with $\dom(F \till a) = a$ which is  dual  to the ideal $(F^-) \till a$. \end{Notation}

\begin{Def} A filter $F$ is called   \emph{$K$-uniform} if for every $a \in F^+$, $F \till a  \leq_K F$. \end{Def}

%
%
 
\begin{Lem} \label{cc} Suppose $G \till a \leq_K F$  for all $a \in G^+$. Then $\G(\ILFF) \Rightarrow \G(\ILGG)$. In particular, this holds if $G$ is $K$-uniform and $G \leq_K F$. \end{Lem} 
 
\begin{proof} Let $A \in \G$ and $T \in \ILGG$ arbitrary. For all $\sigma \in T$ extending $\stem(T)$, let $X_\sigma := \Succ_T(\sigma)$ and fix $\pi_\sigma$ witnessing $G \till X_\sigma \leq_K F$. Define $f': \wlw \to \wlw$ by $f'(\varnothing) := \stem(T)$ and $f'(\tau \cc \left<n\right>) := f'(\tau) \cc \left<\pi_{f'(\tau)}(n)\right>$, and let $f: \ww \to \ww$ be the limit of $f'$. Let $A' := f^{-1}[A]$. Then $A \in \G$, so by assumption there is an $S \in \ILFF$ such that $[S] \subseteq A'$ or $[S] \cap A' = \varnothing$, without loss of generality the former.

\p By assumption, we know that for every $\sigma \in S$ extending $\stem(S)$, $\pi_{f'(\sigma)}[\Succ_S(\sigma)] \in G^+$. To make sure that the image under $f$ is an $\ILGG$-tree, prune $S$ to $S^* \subseteq S$, so that $\stem(S^*) = \stem(S)$, and for all $\sigma \in S^*$ extending $\stem(S^*)$, $\pi_{f'(\sigma)}[\Succ_{S^*}(\sigma)] =  \pi_{f'(\sigma)}[\Succ_{S}(\sigma)]$, and $\pi_{f'(\sigma)} \till \Succ_{S^*}(\sigma)$ is injective. Then $f[S^*]$ is the set of branches through an $\ILFF$-tree, and moreover $f[S^*] \subseteq [T] \cap A$. \end{proof}

\begin{Lem} Suppose $G \till a \leq_K F$  for all $a \in G^+$. Then $\G(\ILF) \Rightarrow \G(\ILGG)$. In particular, if $F$ is $K$-uniform then $\G(\ILF) \Rightarrow \G(\ILFF)$ \end{Lem}

\begin{proof}  Let $A \in \G$ and $T \in \ILGG$ be arbitrary. Let $f$ and $A' := f^{-1}[A]$ be as above. By the same argument, 
it suffices to find $S \in \ILFF$ such that $[S] \subseteq A'$ or $[S] \subseteq A'$. 

\p By assumption, there is an $\ILF$-tree $U$ with $[U] \setminus A' \in \I_\ILF$ or $[U] \cap A' \in \I_\ILF$, without loss of generality the former. Since $\I_\ILF$ is Borel-generated, let $B$ be a Borel $\I_\ILF$-positive set such that $B \subseteq A' \cap [U]$. By Lemma \ref{35} $B$ is also $\D_{F^+}$-positive. But then, by Theorem \ref{dichotomy} there exists an $S \in \ILFF$ such that $[S] \subseteq B$, which completes the proof. \end{proof}

\begin{Lem} Suppose $G \leq_{RK} F$. Then $\G(\ILF) \Rightarrow \G(\ILG)$. 
\end{Lem}

\begin{proof} Let $\pi$ witness $G \leq_{RK} F$ and let $f: \ww \to \ww$ be defined by $f(x)(n) := \pi(x(n))$. Clearly $f$ is continuous in the standard sense. Moreover, we claim the following:

\s {\textbf Claim.} \emph{$f$ is continuous and open as a function from $(\ww, \tau_{\ILF})$ to $(\ww, \tau_{\ILG})$}.

\begin{proof} If $[T]$ is a basic open set in $\tau_\ILG$, then $T \in \ILG$ and so $f^{-1}[T]$ is a union of $\ILF$-trees (one for each $f$-preimage of the stem of $T$), so it is open in $\tau_\ILF$. Conversely, if $[S]$ is basic open in $\tau_\ILF$, then $S \in \ILF$. Although $f[S]$ is not necessarily a member of $\ILG$, we can argue as follows: given $y \in f[S]$, let $x \in [S]$ be such that $f(x) = y$. Then prune $S$ to $S^*$ in a similar way as in the proof of Lemma \ref{cc}, in such a way that the function $\pi$ restricted to $\Succ_{S^*}(\sigma)$ is injective for each $\sigma$ while the image $\pi[\Succ_{S^*}(\sigma)]$ remains unchanged. Moreover, we can do this so that $x \in [S^*]$. Then $f[S^*]$ is indeed an $\ILG$-tree, and moreover $y \in f[S^*] \subseteq f[S]$. Since this can be done for every $y \in f[S]$, it follows that $f[S]$ is open in $\tau_\ILG$. \hfill \renewcommand{\qed}{$\square$ (Claim)}\end{proof}

\p From this, it is not hard to conclude that if $A \in \I_\ILG$ then $f^{-1}[A] \in \I_\ILF$. To complete the proof, let $A \in \G$ and let $O$ be $\tau_\ILG$-open. It suffices to find a non-empty $\tau_\ILG$-open $U \subseteq O$ such that $U \subseteq^* A$ or $U \cap A =^* \varnothing$, where $\subseteq^*$ and $=^*$ refers to ``modulo $\I_\ILG$.

\p  Let $A' := f^{-1}[A]$ and $O' := f^{-1}[O]$. Since $A'$ has the Baire property in $\tau_\ILF$, there is an open $U' \subseteq O'$ such that $U \subseteq^* A'$ or $U \cap A' =^* \varnothing$  (wlog the former) where $\subseteq^*$ and $=^*$ refers to ``modulo $\I_\ILF$''. Then there is a Borel set $B$ such that $B \notin \I_\ILF$ and $B \subseteq A' \cap U'$. Hence $f[B]$ is an analytic subset of $A \cap O$, and by the Claim, $f[B] \notin \I_\ILG$. By the $\tau_\ILG$-Baire property of analytic sets, there is an $\tau_\ILG$-open $U$ such that $U \subseteq^* f[B]$. Hence $U \cap O \subseteq^* A'$, which completes the proof. \end{proof}
 
The relationships established in the above three theorems are summarised in Figure \ref{fig}.

\begin{figure}[here] 
$$
\xymatrix@C=1.5cm@R=1.2cm{
\G(\ILG)\ar@{=>}^{G \ K\text{-uniform}}[rr] & &G(\ILGG) \\  & & & (*)\quad \forall a \in G^+ (G \till a \leq_K F) \\
\G(\ILF) \ar@{=>}^{F \ K\text{-uniform}}[rr] \ar@{=>}^{G \leq_{RK} F}[uu] \ar@{=>}^{(*)}[uurr]& &\ar@{=>}_{(*)}[uu]\G(\ILFF) }  
$$\caption{Implications between the properties for filters $F$ and $G$.}
 \label{fig} \end{figure}
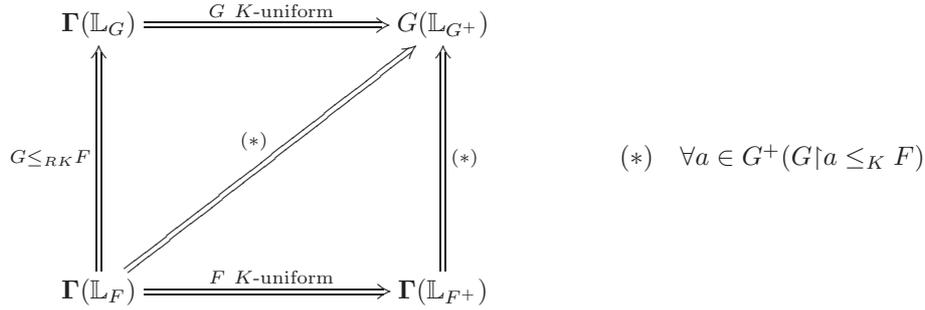

\bigskip In particular, since $\Cof  \till a \leq_K F$ holds for every $F$ and every infinite $a$, we obtain the following corollary:

\begin{Cor} \label{cor1} $\G(\ILF) \Rightarrow \G(\IL)$ and $\G(\ILFF) \Rightarrow \G(\IL)$ for all $F$. \end{Cor}
 
\newcommand{\NWD}{{\sf NWD}} 
 
Next, we look at the relationship between $\ILF$-measurability and the classical Baire property. In accordance to common usage, we denote the statement ``all sets in $\G$ have the Baire property'' by $\G(\IC)$ ($\IC$ denoting the Cohen forcing partial order). It is known that if $F$ is not an ultrafilter then $\ILF$ adds a Cohen real. Specifically, if $Z$ is such that $Z \notin F$ and $(\omega \setminus Z) \notin F$, and $f: \ww \to \dw$ is defined by 
$$f(x)(n) := \left\{ \begin{array}{ll} 1 & \text{ if } x(n) \in Z \\ 0 & \text{ if } x(n) \notin Z \end{array} \right. $$
then $f$ is continuous with the property that if $A$ is meager then $f^{-1}[A] \in \I_\ILF$.

\medskip
Concerning ultrafilters, the following is known.

\begin{Def} \label{nwdultrafilter} Let $\NWD \subseteq \dlw$ denote the ideal of \emph{nowhere dense subsets of $\dlw$}, that is, those $H \subseteq \dlw$ such that $\forall \sigma \; \exists \tau \supseteq \sigma \;\forall \rho \supseteq \tau \: (\rho \notin H)$. An ultrafilter $U$ is called \emph{nowhere dense} iff $\NWD \not\leq_K U^-$. \end{Def}
 
It is known that $\IL_U$ adds a Cohen real iff $U$ is not a nowhere dense ultrafilter. Specifically, if $U$ is not nowhere dense and $\pi:\omega \to \dlw$ is a witness to $\NWD \leq_K U^-$, then we can define a continuous function $f: \ww \to \dw$ by $f(x) := \pi(x(0)) \cc \pi(x(1)) \cc \dots $. We leave it to the reader to verify that if $A$ is meager then $f^{-1}[A] \in \I_{\IL_U}$. This easily leads to the following:

\begin{Lem} \label{lemcohen} If $F$ is not an ultrafilter, or a non-nowhere dense ultrafilter, then $\G(\ILF) \Rightarrow \G(\IC)$. \end{Lem}
 
 \begin{proof} In either case, we have a continuous $f: \ww \to \dw$ such that $f$-preimages of meager sets are $\I_\ILF$-small, as above. Let $A \in \G$ and $\sigma \in \dlw$ arbitrary. Let $\varphi$ be a homeomorphism from $\dw$ to $[\sigma]$ and $A' := (\varphi \circ f)^{-1}[A]$. Then $A' \in \G$, so let $B$ be a Borel $\I_\ILF$-positive set with $B \subseteq A'$ or $B \cap A' = \varnothing$, without loss the former. Then $f[B]$ is an analytic non-meager subset of $A \cap [\sigma]$, so there exists $[\tau] \subseteq [\sigma]$ such that $[\tau]  \subseteq^* A$,  which is sufficient. \end{proof}
 
Finally, an argument from \cite{MillerMesses} yields the following implication. Recall that a set $A \subseteq \wuw$ is \emph{Ramsey} iff there exists $H \in \wuw$ such that $[H]^\omega \subseteq A$ or $[H]^\omega \cap A = \varnothing$.

\begin{Lem} If $U$ is an ultrafilter then $\G(\IL_U) \Rightarrow \G({\rm Ramsey})$. \end{Lem}

\begin{proof} In fact, we prove a stronger statement: if $A \subseteq \wupw$ (strictly increasing sequences) is $\IL_U$-measurable then $\{\ran(x) \mid x \in A\}$ is Ramsey. First note that, by Lemma \ref{b} (4), there exists a $T \in \IL_U$ with empty stem, such that  $[T] \subseteq A$ or $[T] \cap A = \varnothing$. Also, without loss of generality, we can assume that $[T] \subseteq \wupw$.

\p Now proceed inductively: \begin{itemize}

\item Let $n_0 \in \Succ_T(\varnothing)$ be arbitrary.

\item Let $n_1 \in \Succ_T(\varnothing) \cap \Succ_T(\left<n_0\right>)$.

\item Let $n_2 \in \Succ_T(\varnothing) \cap \Succ_T(\left<n_0\right>) \cap  \Succ_T(\left<n_1\right>) \cap  \Succ_T(\left<n_0, n_1\right>) $.

\item etc.
\end{itemize} Since $U$ is a filter we can always continue this process and make sure that for any $k$, any subsequence of the sequence $\left<n_0, \dots, n_k\right>$ is an element of $T$. It then follows that any infinite subsequence  of the sequence $\left<n_i \mid i<\omega\right>$  is an element of $[T]$. This is exactly what we need. \end{proof}
 
If  $U$ is not an ultrafilter, then the above result does not hold in general. For example, considering the cofinite filter, both implications $\G(\IL) \Rightarrow \G($Ramsey$)$ and $\G(\ID) \Rightarrow \G($Ramsey$)$ are consistently false for $\G = \DELTA^1_2$ (see \cite[Section 6]{CichonPaper}).

\section{Analytic filters} \label{Sec5}

In this section, we focus on analytic filters (or ideals). This is important if we want the forcings to be definable, and if we want to apply results from \cite{Ik10, KhomskiiThesis}. Note that just for absoluteness of the forcing, it would have been sufficient to consider $\SIGMA^1_2$ or $\PI^1_2$ filters, by Shoenfield absoluteness. However, we also require the ideals and other related notions to have a sufficiently low complexity. For this reason, in this section the following assumption will hold:

\medskip
\noindent \textbf{Assumption.} $F$ is an analytic filter on $\omega$.



\medskip
It is clear that the  statement ``$T \in \ILF$'' is as complex as $F$ itself. Recall from  \cite[Section 3.6]{BaJu95}) that a forcing notion is \emph{Suslin ccc} if it is ccc and the statements  ``$T \in \ILF$'', ``$T \:\bot\: S$'' and ``$S \leq T$'' are  $\SIGMA^1_1$-relations on the codes of trees. The following is clear:

\begin{Fact} Let $F$ be analytic. Then $\ILF$ is a Suslin ccc forcing notion. \end{Fact}


\begin{Lem} Let $F$ be analytic.  Then the ideals $\I_\ILF$ and $\D_\FF$ are $\SIGMA^1_2$ on Borel sets $($i.e., the membership of Borel sets in the ideal is a $\SIGMA^1_2$-property on the Borel codes$)$. \end{Lem}

\begin{proof} A Borel set $B$ is in $\D_\FF$ iff $\exists \varphi: \wlw \to F \; \forall x \; (x \in B \to \exists^\infty n \:(x(n) \notin \varphi(x\till n)))$. This is easily seen to be a $\SIGMA^1_2$ statement if $F$ is $\SIGMA^1_1$.

\p For $\I_\ILF$, let $B$ be a Borel set. Notice that $B$ is $\tau_{\ILF}$-nowhere dense iff there exists a $\tau_{\ILF}$-open dense set $O$ such that $B \cap O = \varnothing$, iff there is a maximal antichain $A \subseteq \ILF$ such that $B \cap \bigcup\{[T] \mid T \in A\} = \varnothing$. By the ccc, one such maximal antichain can be coded by a real. The resulting computation yields a $\SIGMA^1_2$ statement. \end{proof}

In \cite{BrHaLo, Ik10} the concept of \emph{quasi-generic real} was introduced---a real avoiding all Borel sets in a certain $\sigma$-ideal coded in the ground model. This concept coincides with generic reals for ccc ideals, but yields a weaker concept for other (combinatorial) ideals, see e.g. \cite[Section 2.3]{KhomskiiThesis}.

 In the case of $\ILF$, ``quasi-generic reals'' are the $\ILF$-generic ones, whereas in the case of $\ILFF$, they have a simple characterisation due to the combinatorial ideal $\D_\FF$.

\begin{Lem} \label{quasi} Let $M$ be a model of set theory. A real  $x$ is $\ILF$-generic over $M$ iff $x \notin B$ for every Borel set $B \in \I_\ILF$ with code in $M$. \end{Lem}

\begin{proof} See \cite[Lemma 2.3.2]{KhomskiiThesis}. \end{proof}

\begin{Def} Let $M$ be a model of set theory. We will call a real $x \in \ww$ \emph{$F$-dominating over $M$} if for every $\varphi: \wlw \to F$ with $\varphi \in M$, $x$ $F$-dominates $\varphi$, i.e., $\forall^\infty n \:(x(n) \in \varphi(x\till n))$ (note that the statement $\varphi: \wlw \to F$ is absolute for between $M$ and larger models). \end{Def}

\begin{Lem} Let $M$ be a model of set theory with $\omega_1 \subseteq M$. A real $x$ is $F$-dominating over $M$ iff $x \notin B$ for every Borel set $B \in \D_\FF$ with code in $M$. \end{Lem}

\begin{proof} This is easy to verify from the definition, using $\SIGMA^1_2$-absoluteness between $M$ and $V$ and the fact that $B \in \D_\FF$ is a $\SIGMA^1_2$-statement for Borel sets. \end{proof}

 As an immediate corollary of the above and the general framework from \cite{Ik10} and \cite{KhomskiiThesis}, we immediately obtain the following four characterizations for $(\ILF)$- and $(\ILFF)$-measurability. 

\begin{Cor}  \label{equivalences} Let $F$ be an analytic filter. Then: \begin{enumerate}

\item $\DELTA^1_2(\ILF) \; \Longleftrightarrow \; \forall r \in \ww \; \exists x  \:(x$ is $\ILF$-generic over $L[r])$.

\item $\SIGMA^1_2(\ILF) \; \Longleftrightarrow \; \forall r \in \ww \; \{x \mid x $ not $\ILF$-generic over $L[r]\} \in \I_\ILF$.

\item $\DELTA^1_2(\ILFF) \Longleftrightarrow \forall r \in \ww \; \forall T \in \ILFF \: \exists x \in [T] \:(x$ is $F$-dominating  over $L[r])$.

\item $\SIGMA^1_2(\ILFF) \; \Longleftrightarrow \; \forall r \in \ww \; \{x \mid x $ not $F$-dominating  over $L[r]\} \in \I_\ILFF$.
\end{enumerate}
\end{Cor}
	
\begin{proof} See \cite[Theorem 4.3 and Theorem 4.4]{Ik10} and \cite[Theorem 2.3.7 and Corollary 2.3.8]{KhomskiiThesis}. Note that both ideals $\ILF$ and $\D_\FF$ are $\SIGMA^1_2$, the forcings have absolute definitions and are proper, so the above results can be applied.

\p Only one non-trivial fact requires some explanation. In point 1, the above abstract theorems only yield the statement ``$\DELTA^1_2(\ILF) \; \Longleftrightarrow \; \forall r \in \ww \; \forall T \in \ILF \; \exists x \in [T] \:(x$ is $\ILF$-generic over $L[r])$''. In order to eliminate the clause ``$\forall T \in \ILF$ \dots '', we use the following fact: for every non-principal filter $F$ and every $X \in F$, there exists a bijection $\pi: \omega \to X$, such that for all $a \subseteq X$, $a \in F \; \Leftrightarrow \; \pi^{-1}[a] \in F$. See, e.g., \cite[Lemma 3]{MediniZdomskyy}. We leave it to the reader to verify that this implies homogeneity of $\ILF$, in the sense that if there exists an $\ILF$-generic real then there also exists an $\ILF$-generic real inside $T$ for every $T \in \ILF$. \end{proof}

We are interested in more elegant characterizations of the four above statements.


\begin{Thm} \label{thistheo} $\SIGMA^1_2(\ILF) \; \Longleftrightarrow \; \forall r \in \ww \:(\omega_1^{L[r]} < \omega_1)$. \end{Thm}

The proof uses methods similar to \cite[Theorem 6.2]{LabedzkiRepicky} (see also \cite[Theorem 5.11]{BrLo99}). It follows using a series of definitions and lemmas. 

\newcommand{\rk}{{\rm rk}}
\begin{Def} \label{rankdef} For every open dense set $D \subseteq \ILF$, define a \emph{rank function} $\rk_D: \wlw \to \omega_1$ by \begin{itemize}
\item $\rk_D(\sigma) := 0$ iff there is $T \in D$ with $\stem(T) = \sigma$ and
\item $\rk_D(\sigma) := \alpha$ iff $\rk_D(\sigma) \not<\alpha$ and $\exists Z \in F^+ \; \forall n \in Z \:( \rk_D(\sigma \cc \left<n\right>) < \alpha)$. \end{itemize} \end{Def}

\noindent A standard argument shows that $\rk_D(\sigma)$ is well-defined for every $\sigma$.

\begin{Def} An $(F^-)$-mad family is a collection $\mathcal{A} \subseteq F^+$ such that $\forall a \neq b \in \mathcal{A}$ $(a \cap b) \in F^-$, and $\forall a \in F^+$ there exists $b \in \mathcal{A}$ such that $(a \cap b) \in F^+$. \end{Def}

\begin{Fact} For every analytic filter $F$, there exists an $(F^-)$-mad family of size $2^{\aleph_0}$. \end{Fact}

\begin{proof} See \cite[Corollary 1.8]{KhomskiiBarnabas}. \end{proof}

\begin{Lem} \label{Lefty} Let $\mathcal{A}$ be an $(F^-)$-mad family. For each $a \in \mathcal{A}$, let $X_a := \{x \in \ww \mid \ran(x) \cap a = \varnothing\} \in \N_\ILF$. Then, for any $X \in \I_\ILF$, the collection $\{a \in \mathcal{A} \mid X_a \subseteq X\}$ is at most countable. \end{Lem}

\begin{proof}   Let $X \subseteq \bigcup_n X_n$ where $X_n$ are closed nowhere dense in $\tau_\ILF$, and let $D_n := \{T \mid [T] \cap X_n = \varnothing\}$. Then the $D_n$ are open dense in $\ILF$. Consider a countable elementary submodel $N$ of some sufficiently large $\mathcal{H}_\theta$ containing $\mathcal{A}$, the $D_n$, and the defining parameter of $F$ (i.e., the $r \in \ww$ such that $F \in \Sigma^1_1(r)$). The proof will be completed by showing that if $a \in \mathcal{A} \setminus N$, then there exists $x \in X_a \cap \bigcap_n \bigcup \{[T] \mid T \in D_n\}$, hence $x \in X_a \setminus X$.

\medskip
 \s{Sublemma.} For every $D_n$, every  $a \in \mathcal{A} \setminus N$, and every $T \in \ILF$, if  $\ran(\stem(T)) \cap a = \varnothing$ then there exists $S \leq T$ with $S \in D_n$ and such that $\ran(\stem(S)) \cap a = \varnothing$ as well.

\begin{proof} Let $Y := \{\tau \in T \mid \stem(T) \subseteq \tau$ and $\ran(\tau) \cap a = \varnothing\}$. Let $\tau \in Y$ be of least $D_n$-rank. We claim that $\rk_{D_n}(\tau) = 0$, which completes the proof. Towards contradiction, assume $\rk_{D_n}(\tau) = \alpha > 0$ and let $Z \in F^+$ witness this. By elementarity and using the fact that all relevant objects are in $N$ and $F$ is absolute for $N$ as well, it follows that $Z \in N$. 

\p By elementarity and absoluteness of $F$, $N \models $ ``$\mathcal{A}$ is an $(F^-)$-mad family'', hence there exists $b \in \mathcal{A} \cap N$ such that $Z \cap b \in F^+$. Since $b \neq a$, it follows that $b \cap a \in F^-$, so there exists $n \in (Z \setminus a)$. Then $\tau \cc \left<n\right>$ is an element of $Y$ with $D_n$-rank less than $\alpha$, contradicting the minimality of $\tau$.  \renewcommand{\qed}{\hfill $\Box$ (Sublemma)} \end{proof}

\p Now, it is clear that we can inductively apply the sublemma to find a sequence $T_0 \geq T_1 \geq T_2 \geq \dots$, with strictly increasing stems, such that $T_n \in D_n$ for every $n$, and moreover $\ran(\stem(T_n)) \cap a = \varnothing$ for every $n$. Then $x := \bigcup_n \stem(T_n)$ has all the required properties, i.e., $x \in X_a \setminus X$.  \end{proof}

\begin{proof}[Proof of Theorem \ref{thistheo}] We need to prove the equivalence between \begin{enumerate}
\item $ \forall r \;  \{x \mid x $ not $\ILF$-generic over $L[r]\} \in \I_\ILF$ and 
\item $ \forall r \;(\omega_1^{L[r]} < \omega_1)$. \end{enumerate}
By Lemma \ref{quasi}, the former statement is equivalent to $\forall r \; \bigcup\{B \mid B$ is a Borel $\I_\ILF$-small set with code in $L[r]\} \in \I_\ILF$. The direction from 2 to 1 is thus immediate.

\p  Conversely, fix $r$ and assume that $\omega_1^{L[r]} = \omega_1$. Let $\mathcal{A}$ be an $(F^-)$-mad family such that $| \mathcal{A} \cap L[r] | = \omega_1$ (this can be done by extending an $(F^-)$-almost disjoint family of size $\omega_1$ in $L[r]$). 
For every $a \in \mathcal{A} \cap L[r]$, $X_a$ is a Borel $\I_\ILF$-small set with code in $L[r]$. If 1 was true, then in $V$ there would be an $X \in \I_\ILF$ such that $X_a \subseteq X$ for all such $a$, contradicting Lemma \ref{Lefty}. \end{proof}

\begin{Remark} The same argument yields $\add(\I_\ILF) = \omega_1$ and $\cof(\I_\ILF) = \cont$ for analytic filters (where $\add$ and $\cof$ denote the \emph{additivity} and \emph{cofinality} numbers of the ideal, respectively). \end{Remark}

\medskip  Next, we consider $\DELTA^1_2(\ILF)$ and $\DELTA^1_2(\ILFF)$.
In \cite[Theorem 2]{BrendleUltrafilters}, the covering number of $\I_{\IL_U}$ for an ultrafilter $U$ was determined to be the minimum of $\bb$ and a certain combinatorial characteristic of $U$ called $\pi \mathfrak{p}(U)$. This was generalised by Hrusak and Minami in \cite[Theorem 2]{HrusakMinami} to arbitrary filters. Similar proofs yield characterisations of $\DELTA^1_2(\ILF)$ and $\DELTA^1_2(\ILFF)$.


\begin{Def} Let $M$ be a model of set theory and $F$ an analytic filter. We say that a real $C \in [\omega]^\omega$ is \begin{enumerate}

\item \emph{$F$-pseudointersecting over $M$} if $C \subseteq^* a$ for all $a \in F \cap M$.
\item \emph{$F$-separating over $M$} if it is $F$-pseudointersecting over $M$, and additionally, for all $b \in (F^+) \cap M$, $|C \cap b| = \omega$.\end{enumerate}
We use the shorthand  ``$\exists F$-${\sf pseudoint}$'' and ``$\exists F$-${\sf sep}$'' to abbreviate the statements ``$\forall r \in \ww \: \exists C \; (C$ is $F$-pseudointersecting/separating over $L[r]$)''. \end{Def}

\begin{Question} Are there natural regularity properties equivalent to  ``$\exists F$-${\sf pseudoint}$'' and ``$\exists F$-${\sf sep}$'' for $\DELTA^1_2$ sets of reals? \end{Question}

Recall that $\SIGMA^1_2(\IC)$ is equivalent to $\DELTA^1_2(\IC) \land \DELTA^1_2(\IL)$ and equivalent to $\DELTA^1_2(\ID)$, where $\IC, \IL$ and $\ID$ stand for the Baire property, Laver- and Hechler-measurability, respectively. Also, recall that $\DELTA^1_2(\IC)$ is equivalent to the existence of Cohen reals over $L[r]$, $\DELTA^1_2(\IL)$ is equivalent to the existence of dominating reals over $L[r]$, and $\DELTA^1_2(\ID)$ is equivalent to the existence of Hechler-generic reals over $L[r]$. See \cite[Theorem 4.1 and Theorem 5.8]{BrLo99}. Also, note that $\ID$ and $\IL$ are just $\IL_\Cof$ and $\IL_{\Cof^+}$.


\begin{Thm} $\DELTA^1_2(\ILF) \; \Longleftrightarrow \; \SIGMA^1_2(\IC) \: \land \: \exists F$-${\sf sep}$. \end{Thm}

\begin{proof} By Corollary \ref{cor1} and Lemma \ref{lemcohen}, we know that $\DELTA^1_2(\ILF)$ implies $\DELTA^1_2(\IL)$ and $\DELTA^1_2(\IC)$, which in turns implies  $\SIGMA^1_2(\IC)$ as mentioned above.  Moreover, a standard density argument shows that if $\ILF$ generically adds an $F$-separating real, specifically, if $x$ is $\ILF$-generic  then $\ran(x)$ is $F$-separating.

\p For the converse direction, fix $r \in \ww$ and let $C$ be $F$-separating over $L[r]$. Let $\ID_C$ denote Hechler forcing as defined on $C^\omega$ (i.e., the conditions are trees in $C^{<\omega}$ with branching into all of $C$ except for finitely many points). Clearly $\ID_C$ is isomorphic to the ordinary Hechler forcing. Notice that for every $T \in \ILF$, if $\ran(\stem(T)) \subseteq C$ then $T \cap C^{<\omega} \in \ID_C$. 

\p For every $D \in L[r]$ dense in $\ILF$, let $D' := \{T \cap  C^{<\omega} \mid T \in D$ and $\ran(\stem(T)) \subseteq C\}$. We claim that $D'$ is predense in $\ID_C$. Let $S \in \ID_C$ be arbitrary, with $\sigma := \stem(S)$. Recall the rank-function from Definition \ref{rankdef}. Since $D \in L[r]$, we consider the rank function $\rk_D$ as defined inside $L[r]$. If $\rk_D(\sigma) = 0$ then there is $T \in D$ with $\stem(T) = \sigma$, hence $S$ and $T$ are compatible. Otherwise, let $\rk_D(\sigma) = \alpha$. By definition of $\rk_D$ and the fact that $\rk_D$ is in $L[r]$, there exists $Z \in F^+$ with $Z \in L[r]$, such that $\rk_D(\sigma \cc \left<n\right>) < \alpha$ for all $n \in Z$. Since $\Succ_S(\sigma) \cap Z$ is also in $\F^+$, by assumption, there is $n \in C \cap \Succ_S(\sigma) \cap Z$. Continuing  this process, we arrive at some $\tau$ extending $\sigma$, such that $\tau \in S$, $\ran(\tau) \subseteq C$ and $\rk_D(\tau) = 0$. Then we are done as before.

\p By the remark above, $\SIGMA^1_2(\IC)$ implies $\DELTA^1_2(\ID)$, which implies the existence of Hechler-generic reals. In particular, there is a $d \in C^\omega$ which is $\ID_C$-generic real over $L[r][C]$. But then $d$ is $\ILF$-generic over $L[r]$, since for every $D \in L[r]$ dense in $\ILF$, we find $T \in D$ with $d \in [T \cap  C^{<\omega}]$. \end{proof}

A similar argument can be used to simplify $\DELTA^1_2(\ILFF)$; however, here the \emph{homogeneity} of $\ILFF$ provides an additional obstacle, since $\ILFF$ is, in general, only homogeneous if $F$ is $K$-uniform. 

\begin{Thm} $\DELTA^1_2(\ILFF)\;  \Longrightarrow \; \DELTA^1_2(\IL) \land \exists F$-${\sf pseudoint}$. If $F$ is $K$-uniform, then the converse implication holds. \end{Thm}

\begin{proof} By Corollary \ref{cor1} we know that $\DELTA^1_2(\ILFF)  \Rightarrow \DELTA^1_2(\IL)$. Let $x$ be $F$-dominating over $L[r]$ and let $C := \ran(x)$. For each  $a \in F \cap L[r]$  let $\varphi$ be the function given by $\varphi(\sigma) := a \setminus |\sigma|$ for all $\sigma \in \wlw$.  Since $\forall^\infty n \:(x(n) \in \varphi(x \till n))$, clearly $C$ is infinite and $C \subseteq^* a$.

\p Conversely, assume that $F$ is $K$-uniform. We leave it to the reader to verify that, if $T \in \ILFF$, then there exists a continuous function $f: \ww \to [T]$ such that $f$-preimages of $\D_\FF$-small sets are $\D_\FF$-small. In particular, the statements \begin{itemize}

\item $\exists x \: (x$ is $F$-dominating over $L[r])$, and
\item for all $T \in \ILFF \; \exists x \in [T] \;  (x $ is $F$-dominating over $L[r])$ \end{itemize}
are equivalent for all $T \in \ILFF \cap L[r]$.

\p So, fix $r \in \ww$ and let $T \in \ILFF$ be arbitrary. Without loss of generality, assume $T \in L[r]$ (otherwise, use $L[r][T]$ instead). Let $C$ be $F$-pseudointersecting over $L[r]$. For each $\varphi: \wlw \to F$ from $L[r]$, define $g_\varphi: \wlw \to \omega$ by $g_\varphi(\sigma) := \min \{ n \min C \setminus n \subseteq \varphi(\sigma)\}$. Then all $g_\varphi$ are in $L[r]$, by $\DELTA^1_2(\IL)$ there is a dominating real $g$ over $L[r]$, so, in particular, $g$ dominates all $g_\varphi$. Let $x \in \ww$ be such that $x(n) \in C$ and $x(n) \geq g(x \till n)$ for every $n$. Clearly for every $\varphi \in L[r]$ we have $\forall^\infty n \; (x(n) \in \varphi(x \till n))$, hence $x$ is $F$-dominating over $L[r]$. This suffices by what we mentioned above.
\end{proof}

Currently, we do not have a similarly elegant characterization for $\SIGMA^1_2(\ILFF)$. 

\begin{Question} Is there a characterization of $\SIGMA^1_2(\ILFF)$ similar to the above? Is $\SIGMA^1_2(\ILF)$ equivalent to $\DELTA^1_2(\ILFF)$? \end{Question}

\bibliographystyle{alpha}
\bibliography{../../10_Bibliography/Khomskii_Master_Bibliography}{}

\end{document}